\documentclass[12pt]{amsart}
\usepackage{amsmath,amssymb,amsbsy,amsfonts,latexsym,amsopn,amstext,cite,
                                               amsxtra,euscript,amscd,bm,mathabx,mathrsfs}
\usepackage{url}
\usepackage[colorlinks,linkcolor=blue,anchorcolor=blue,citecolor=blue,backref=page]{hyperref}
\usepackage{color}
\usepackage{graphics,epsfig}
\usepackage{graphicx}
\usepackage{float} 
\usepackage[english]{babel}
\usepackage{mathtools}
\usepackage{todonotes}
\usepackage{url}
\usepackage[colorlinks,linkcolor=blue,anchorcolor=blue,citecolor=blue,backref=page]{hyperref}

\usepackage[norefs,nocites]{refcheck}

\hypersetup{breaklinks=true}

\usepackage[norefs,nocites]{refcheck}
\usepackage[english]{babel}

\begin{document}

\newtheorem{theorem}{Theorem}
\newtheorem{lemma}[theorem]{Lemma}
\newtheorem{claim}[theorem]{Claim}
\newtheorem{cor}[theorem]{Corollary}
\newtheorem{prop}[theorem]{Proposition}
\newtheorem{definition}{Definition}
\newtheorem{question}[theorem]{Open Question}
\newtheorem{conj}[theorem]{Conjecture}
\newtheorem{prob}{Problem}

\numberwithin{equation}{section}
\numberwithin{theorem}{section}
\numberwithin{table}{section}
\def\ssssum{\mathop{\sum\!\sum\!\sum\!\sum}}
\def\sssum{\mathop{\sum\!\sum\!\sum}}
\def\ssum{\mathop{\sum \sum}}
\def\dsum{\mathop{\sum \sum}}
\def\iint{\mathop{\int\ldots \int}}

\def\squareforqed{\hbox{\rlap{$\sqcap$}$\sqcup$}}
\def\qed{\ifmmode\squareforqed\else{\unskip\nobreak\hfil
\penalty50\hskip1em\null\nobreak\hfil\squareforqed
\parfillskip=0pt\finalhyphendemerits=0\endgraf}\fi}

\def\cA{{\mathcal A}}
\def\cB{{\mathcal B}}
\def\cC{{\mathcal C}}
\def\cD{{\mathcal D}}
\def\cE{{\mathcal E}}
\def\cF{{\mathcal F}}
\def\cG{{\mathcal G}}
\def\cH{{\mathcal H}}
\def\cI{{\mathcal I}}
\def\cJ{{\mathcal J}}
\def\cK{{\mathcal K}}
\def\cL{{\mathcal L}}
\def\cM{{\mathcal M}}
\def\cN{{\mathcal N}}
\def\cO{{\mathcal O}}
\def\cP{{\mathcal P}}
\def\cQ{{\mathcal Q}}
\def\cR{{\mathcal R}}
\def\cS{{\mathcal S}}
\def\cT{{\mathcal T}}
\def\cU{{\mathcal U}}
\def\cV{{\mathcal V}}
\def\cW{{\mathcal W}}
\def\cX{{\mathcal X}}
\def\cY{{\mathcal Y}}
\def\cZ{{\mathcal Z}}

\def\sfH{\mathsf H} 
\def\sH{\mathscr H} 

\def\MNL{{\mathfrak M}(N;K,L)}
\def\VNL{V_m(N;K,L)}
\def\RNL{R(N;K,L)}

\def\MNm{{\mathfrak M}_m(N;K)}
\def\VNm{V_m(N;K)}

\def\Xm{\cX_m}

\def \C {{\mathbb C}}
\def \F {{\mathbb F}}
\def \L {{\mathbb L}}
\def \K {{\mathbb K}}
\def \Q {{\mathbb Q}}
\def \R {{\mathbb R}}
\def \Z {{\mathbb Z}}

\def \Fp {{\mathbb F}_p}

\def\\{\cr}
\def\({\left(}
\def\){\right)}
\def\fl#1{\left\lfloor#1\right\rfloor}
\def\rf#1{\left\lceil#1\right\rceil}

\def \Prob{{\mathrm {}}}
\def\e{\mathbf{e}}
\def\ep{\e_p}
\def\Res{\mathrm{Res}}
\def\mand{\qquad \text{and} \qquad}
\def \li {\mathrm {li}\,}

\newcommand{\comm}[1]{\marginpar{%
\vskip-\baselineskip 
\raggedright\footnotesize
\itshape\hrule\smallskip#1\par\smallskip\hrule}}

\title{Multiplicative Properties of Hilbert Cubes}

 \author[I.~E.~Shparlinski]{Igor E. Shparlinski}
 \address{School of Mathematics and Statistics, University of New South Wales.
 Sydney, NSW 2052, Australia}
 \email{igor.shparlinski@unsw.edu.au}
\date{}

\begin{abstract} We obtain upper bounds on the cardinality of Hilbert cubes in 
finite fields, which avoid large product sets and reciprocals of sum sets.  In particular,
our results replace recent estimates of N.~Hegyv{\'a}ri and  P.~P.~Pach (2020), which 
appear to be void for all admissible parameters. Our approach is different from that 
of N.~Hegyv{\'a}ri and  P.~P.~Pach and is based on  some well-known bounds
of double character and exponential sums
over arbitrary sets, due to A. A. Karatsuba (1991) and N. G. Moshchevitin (2007), 
respectively. 
\end{abstract}


\maketitle

\section{Introduction}

\subsection{Set-up and background}

Let $p$ be a prime. Given  $d+1$ elements  $a_0, a_1, \ldots, a_d \in \Fp$ of a finite field $\F_p$ 
of $p$ elements,  the corresponding {\it  Hilbert cube\/} of {\it dimension\/} $d$   is defined as
the following set 
\begin{equation}\label{eq:Haa}
\sfH\(a_0; a_1, \ldots, a_d\) =\left \{a_0+\sum_{i=1}^d \varepsilon_ia_i : ~\varepsilon_i \in \{0,1\}\right\}.
\end{equation}

Since Hilbert cubes have a very strong additive structure, it is natural to study 
how  their interact with (for example, avoid) multiplicatively defined sets, such 
as product sets 
$$
\cA\cdot \cB = \{ab:~ a\in \cA, \ b \in \cB\}
$$
and reciprocal sum sets
$$
\(\cA + \cB\)^{-1} = \{(a+b)^{-1} :~ a\in \cA, \ b \in \cB, \ a\ne - b\}
$$
for $\cA, \cB \subseteq \Fp$.

In particular, Hegyv{\'a}ri and  Pach~\cite{HePa} have considered the question of estimating 
the cardinality $\# \cH$ of Hilbert cubes $\cH$ (in terms of $\# \cA$ and $\# \cB$), 
such that $\cH \cap\( \cA\cdot\cB\) =\emptyset$ 
or $\cH \cap \(\cA + \cB\)^{-1} =\emptyset$ for some sets  $\cA, \cB \subseteq \Fp$. 
However, unfortunately the bounds of~\cite[Theorems~3.1 and~3.2]{HePa} 
never improve the trivial bound $\# \cH \le p$. Here we start with  a brief analysis of 
these results. 

First we notice that it appears that in the formulation of~\cite[Theorem~3.1]{HePa} 
the exponent $13/16$ has to be $5/16$, see the last displayed equation at the end 
of the proof  of~\cite[Theorem~3.1]{HePa}. 
Then the  bound  of~\cite[Theorem~3.1]{HePa} takes the following form.
If $\cH \cap\( \cA\cdot\cB\) =\emptyset$  for some sets $\cA \subseteq \Fp$ and $\cB \subseteq \cG$, 
where $\cG\subseteq \Fp^*$  is a multiplicative subgroup of order $\# \cG = O\(p^{3/4}\)$, then 
\begin{equation}
\label{eq:HP Triv1} 
\# \cH \le  c \frac{p^{9/4}\(\# \cG\)^{1/4}}{\(\#\cA\)^{1/2-\eta/4}\(\#\cB\)^{5/8}},
\end{equation}
with some absolute constant $c>0$, where the parameter $\eta>0$ controls the 
additive structure of $\cA$ (its additive energy);  we refer to the formulation
of~\cite[Theorem~3.1]{HePa}  for more details. 

Since $\eta >0$, $\#\cA \le p$, $\#\cB \le \#\cG  \le p^{3/4}$ (see~\cite[Theorem~3.1]{HePa} 
for more details on these parameters), 
we conclude that the right hand side of~\eqref{eq:HP Triv1}  satisfies 
\begin{align*}
 \frac{p^{9/4}\(\# \cG\)^{1/4}}{\(\#\cA\)^{1/2-\eta/4}\(\#\cB\)^{5/8}} &
 \ge   \frac{p^{9/4}\(\# \cG\)^{1/4} }{p^{1/2}\(\#\cG\)^{5/8}}\\
& =  \frac{p^{9/4}}{p^{1/2}\(\#\cG\)^{3/8}}
  \ge   \frac{p^{9/4}}{p^{1/2+9/32}} = p^{47/32},
\end{align*}
substantially exceeding the trivial bound $\# \cH \le  p$. 

Furthermore, the bound~\cite[Theorem~3.2]{HePa}
takes form
\begin{equation}
\label{eq:HP Triv2} 
\# \cH \le  16\frac{p^{3}}{\#\cA\#\cB}. 
\end{equation}
Since $\cA, \cB \subseteq \Fp$, the right hand side of~\eqref{eq:HP Triv2}
is at least $32p$, which again exceeds the trivial bound $\# \cH \le  p$.

We also notice that a series of very interesting results on product $\cH \cH$ and ratio  $\cH \cH^{-1}$  sets
of Hilbert cubes have recently been given by Shkredov~\cite{Shkr}. 
Moreover, Shkredov~\cite{Shkr} has also given dual results on sum sets 
of multiplicatively defined analogues of Hilbert cubes. 

\subsection{Notation and conventions}
Throughout the paper, the notations $U = O(V)$, 
$U \ll V$ and $ V\gg U$  are equivalent to $|U|\leqslant c V$ for some positive constant $c$, 
which may depend on the integer positive parameter $r$.

To avoid complicated expressions, especially in the exponents, we deviate from the canonical 
rule and  write fractions of the form $1/km$ to mean $1/(km)$ 
(rather than $(1/k)\cdot m$).

\subsection{New results}

Here we obtain  new results on the size of Hilbert cubes which avoid product sets
and reciprocal sum sets. Our method is different from that of Hegyv{\'a}ri and  Pach~\cite{HePa} 
and is based on some bounds of double character and exponential sums.

In particular, we unify both events $\cH \cap (\cA\cdot\cB) = \emptyset$ and $\cH \cap \(\cA + \cB\)^{-1} =\emptyset$ in the following result.

\begin{theorem}\label{thm:H-Small}
    Let $\cA, \cB \subseteq \Fp^*$ and let $\cH \subseteq \Fp$ be a Hilbert
    cube with 
    $$\cH \cap (\cA\cdot\cB) = \emptyset\qquad \text{or}
    \qquad \cH \cap \(\cA + \cB\)^{-1} =\emptyset.
    $$ 
    Then for any positive
    integer  $r$  
    we have
$$
\#\cH \ll   \frac{p^{2r + 1/2+1/2r}}{\(\#\cA\#\cB\)^r} .
$$
\end{theorem}

Taking Theorem~\ref{thm:H-Small} with $r = 2$ we see that for 
$$
\#\cA\#\cB\gg p^{15/8+\varepsilon/2}
$$
with an arbitrary $\varepsilon>0$ we obtain a nontrivial upper bound 
$$\#\cH \ll  p^{1-\varepsilon}.
$$
Furthermore, for an  arbitrary $\varepsilon>0$  and $r = \rf{4/\varepsilon}$ we derive from
Theorem~\ref{thm:H-Small}  that for 
$$
\#\cA\#\cB\gg p^{2-1/(4r^2)}
$$
we have 
$$\#\cH \ll  p^{1/2 + \varepsilon}.
$$

\subsection{Applications}

As we have mentioned,  Shkredov~\cite{Shkr} has studied the product $\cH \cH$ and ratio  $\cH \cH^{-1}$  sets
of Hilbert cubes and in particular has given lower bounds on $\#\(\cH \cdot \cH\)$ 
and  $\#\(\cH \cdot \cH^{-1}\)$, 
provided that $\# \cH$ is not too large. We now observe that Theorem~\ref{thm:H-Small} implies the 
following complementing  result,  which applies for large Hilbert cubes and shows that their product and ratio 
sets occupy almost all $\F_p$. 

\begin{cor}\label{cor:HH H/H}
Let  $\cH \subseteq \Fp$ be a Hilbert cube and let 
$$
\cE =  \F_p \setminus \(\cH\cdot \cH\) \mand 
\cF =  \F_p \setminus \(\cH  \cdot\cH^{-1}\) .
$$ 
Then 
$$
\#\cE, \#\cF \ll  p^{19/8} \(\# \cH\)^{-3/2}. 
$$
\end{cor}

Indeed, to see Corollary~\ref{cor:HH H/H}, it is enough to notice  that 
$$
\cH \cap (\cE\cdot\cH^{-1}) = \emptyset \mand \cH \cap (\cF\cdot\cH) = \emptyset
$$ 
and apply Theorem~\ref{thm:H-Small} with $r=2$ to 
$$\(\cA, \cB\) = \(\cE, \cH^{-1}\) \mand \(\cA, \cB\) = \(\cF, \cH\).
$$ 
deriving 
$$
\#\cH \ll   \frac{p^{4 + 1/2+1/4}}{\(\#\cF\#\cH\)^2}
$$
in each case.

Corollary~\ref{cor:HH H/H} is nontrivial if $ \#\cH \gg  p^{11/12+\varepsilon}$ with 
some fixed  $\varepsilon>0$. 

In turn, Corollary~\ref{cor:HH H/H} immediately implies the following. 

\begin{cor}\label{cor:H Soubgr} There is an absolute constant $C$ such that if 
$\cH \subseteq \Fp$ be a Hilbert cube with 
$$
\#\cH  \ge C  p^{11/12}
$$
then $\cH$ is not contained in a co-set $\lambda \cG$, $\lambda \in \F_p^*$, of any proper multiplicative subgroup $\cG \subseteq \F_p^*$. 
\end{cor}

Indeed, if  $\cH\subseteq \lambda \cG$ then  $\cH\cdot \cH \subseteq \lambda^2 \cG$. 
Hence 
$$
\#\( \F_p \setminus \(\cH\cdot \cH\)\) \ge \#\( \F_p \setminus  \lambda^2 \cG\) \le p - (p-1)/2 = (p+1)/2
$$
and by Corollary~\ref{cor:HH H/H}  we obtain the bound of Corollary~\ref{cor:H Soubgr}. 

Corollary~\ref{cor:H Soubgr} complements a series of results in~\cite{AlSh,DES1,DES2,HeSa}  on Hilbert cubes $\cH$ avoiding the sets 
of quadratic non-residues and primitive roots of $\F_p$, which characterise $\cH$ in terms of its dimension $d$ of $\cH$ rather than of its size $\#\cH$. 

 For example,  let $f(p)$ and $F(p)$ be the largest dimension of $\cH$ such that $\cH$ does not 
 contain quadratic non-residues and primitive roots of $\F_p$, respectively. 
 
 Hegyv{\'a}ri and  S{\'a}rk{\"o}zy~\cite[Theorem~2]{HeSa} have proved that $f(p) \le 12p^{1/4}$.
 This has been improved as $F(p) \le p^{1/5+o(1)}$ in~\cite[Theorem~1.3]{DES1} and 
 then as $F(p) \le p^{3/19+o(1)}$ in~\cite[Theorem~1.3]{DES2}. 
Very recently, Alsetri and Shao~\cite{AlSh} have given a further improvement 
and established the bound $F(p) \le p^{1/8+o(1)}$, which is the best  possible estimate 
until the celebrated bound of Burgess~\cite{Burg} on the smaller primitive roots $g(p) \le p^{1/4+o(1)}$
is improved.

\section{Preliminaries}

\subsection{Bounds of double character and exponential sums}

Let $\cX$ be the set of all $p-1$ multiplicative characters of $\Fp$ and 
let $\cX^* = \cX\setminus \{\chi_0\}$ be the set of all  non-principal characters,
 see~\cite[Chapter~3]{IwKow} for a background.

We first recall the following result of Karatsuba~\cite{Kar1}
(see also~\cite[Chapter~VIII, Problem~9]{Kar2}): 

\begin{lemma} \label{lem: Double Char Sum}
    Let $\cU, \cV \subseteq \Fp$ be of cardinalities  $U$ and $V$,
    respectively.  Then for any positive  integer  $r$, 
    for all   $\chi\in \cX^*$, we have
$$
        \sum_{u \in \cU} \sum_{v \in \cV} \chi\(u+v\) \ll U^{1-1/2r}\(V^{1/2}p^{1/2r} + Vp^{1/4r}\),
$$
    where the implied constant depends only on $r$.
\end{lemma}

%

Furthermore, Moschevitin~\cite[Theorem~4]{Mosh} has given an additive analogue  of Lemma~\ref{lem: Double Char Sum} for exponential sums with 
reciprocals.

We denote $\ep(z) = \exp\(2 \pi i z/p\)$.

\begin{lemma} \label{lem: Double Exp Sum}
    Let $\cU, \cV \subseteq \Fp$ be of cardinalities  $U$ and $V$,
    respectively.  Then for any positive  integer  $r$, 
    for all    $\lambda \in \Fp^*$, we have
$$
        \ssum_{\substack{u \in \cU, \ v \in \cV\\ u+v \ne 0}} \ep\(\frac{\lambda}{u+v}\) \ll U^{1-1/2r}\(V^{1/2}p^{1/2r} + Vp^{1/4r}\),
$$
    where the implied constant depends only on $r$.
\end{lemma}

We also note that in some cases the bound of~\cite[Lemma~6]{BBS}
(which has also been repeated as~\cite[Lemma~4.3]{HePa}) 
$$
        \ssum_{\substack{u \in \cU, \ v \in \cV\\ u+v \ne 0}} \ep\(\frac{\lambda}{u+v}\) \le
\sqrt{pUV}
$$
is stronger,  however, it is not useful for the problems of this paper. 

\subsection{Partitioning Hilbert cubes}

We need the following elementary statement, which shows that any Hilbert cube can be represented as a sum set
of two sets of essentially any desired size.

\begin{lemma} \label{lem: Part Hilb}
Let $\cH$ be a Hilbert cube of cardinality  $H$. For any real   $R \in [2, H/2]$ there are two sets 
$\cU$ and $\cV$ such that 
$$
    \cH = \cU + \cV
$$ 
and
$$
 \# \cU \ge H/R \mand \# \cV \ge  R/2.
$$
\end{lemma}

\begin{proof} We recall the notation~\eqref{eq:Haa} and assume that 
$$
\cH = \sfH(a_0; a_1, \ldots, a_d)
$$ 
for some $a_0, a_1, \ldots, a_d \in \Fp$.   We consider the sequences of Hilbert cubes
$$
\cH_j =  \sfH\(a_0; a_1, \ldots, a_j\), \qquad j = 0, \ldots, d.
 $$
 Clearly $\cH_0 = \{a_0\}$ and $\cH_d = \cH$. 
Let $i$ be the smallest $j$ for which $\# \cH_j\ge H/R$.
In particular $i \ge 1$. Since by the choice of $i$ we have 
$ \#  \cH_{i-1} < H/R$ and also since 
$$
\cH_{i} = \cH_{i-1}  \cup \(\cH_{i-1} + a_i\)
$$
we trivially have $\# \cH_j \le2 \#  \cH_{i-1} < 2H/R\le H$.
We now set  
$$
\cU = \cH_{i}  \mand \cV =  \sfH\(0; a_{i+1} , \ldots, a_d\).
$$
Thus $2H/R >  \# \cU  \ge H/R$. We also  have $\cH = \cU + \cV$ and hence
$ H  \le \# \cU \#\cV < 2H R^{-1} \#\cV$, which concludes the proof. 
\end{proof}

\section{Proof of Theorem~\ref{thm:H-Small}}

\subsection{Bounding the size of Hilbert cubes avoiding product sets}
\label{sec: H AB} 
Assume that  for some sets $\cA, \cB \subseteq \Fp^*$ and a Hilbert cube $\cH \subseteq \Fp$  
we have $\cH \cap (\cA\cdot\cB) = \emptyset$.  

We fix some $r$ and  set $R =p^{1/2r}$.
We also denote 
$$
A= \# \cA, \qquad B= \# \cB, \qquad H =\# \cH.
$$ 

Since $AB \le p^2$, the result is trivial for $H < 2R$.
We can also assume that $p\ge 2^{2r}$ and thus $R \ge 2$. Hence,  we see that the conditions of Lemma~\ref{lem: Part Hilb}
are satisfied. Let $\cU$ and $\cV$ be the corresponding sets.  The condition $\cH \cap (\cA\cdot\cB) = \emptyset$
implies that the number of solutions $N$ to the equation
\begin{equation}
\label{eq:Sum Prod} 
u+v =ab, \qquad a\in \cA, \ b \in \cB, \ u \in \cU, \ v \in \cV,
\end{equation}
satisfies 
\begin{equation}
\label{eq:Vanish} 
N = 0.
\end{equation}

On the other hand, using the orthogonality of characters (and the convention $\chi(0) =0$) we write
\begin{align*}
N &= \sum_{a\in \cA}  \sum_{b \in \cB}  \sum_{u \in \cU}  \sum_{v \in \cV}
\frac{1}{p-1}\sum_{\chi\in \cX}  \chi\((u+v) a^{-1} b^{-1}\)\\
&= \frac{1}{p-1}\sum_{\chi\in \cX}   \sum_{a\in \cA}  \chi\( a^{-1}  \) \sum_{b \in \cB}  
 \chi\(b^{-1}\)\sum_{u \in \cU}  \sum_{v \in \cV} \chi\(u+v\).
\end{align*}
We note that the contribution from the principal character $\chi_0$ is given by 
\begin{align*}
\sum_{\chi\in \cX}   \sum_{a\in \cA}  \chi_0\( a^{-1}  \) \sum_{b \in \cB}  &
 \chi_0\(b^{-1}\)\sum_{u \in \cU}  \sum_{v \in \cV} \chi_0\(u+v\)\\
 & = AB  \ssum_{\substack{u \in \cU, \ v \in \cV\\ u+v \ne 0}} 1 = AB  \(UV + O(U)\) \gg AB UV,
\end{align*}
which together with the vanishing condition~\eqref{eq:Vanish} implies 
\begin{align*}
ABUV&  \ll 
\sum_{\chi\in \cX^*}  \left| \sum_{a\in \cA}  \chi\( a^{-1}  \) \right|  \left|  \sum_{b \in \cB}  
 \chi\(b^{-1}\)\right|  \left| \sum_{u \in \cU}  \sum_{v \in \cV} \chi\(u+v\) \right|\\
 & =  \sum_{\chi\in \cX^*}  \left| \sum_{a\in \cA}  \chi\( a  \) \right|  \left|  \sum_{b \in \cB}  
 \chi\(b\)\right|  \left| \sum_{u \in \cU}  \sum_{v \in \cV} \chi\(u+v\) \right|,
\end{align*}
since $ \chi\( z^{-1}  \)=\overline \chi\( z \)$ for any $z \in \F_p^*$. 

Applying  Lemma~\ref{lem: Double Char Sum} and recalling the choice of $R$, we derive 
\begin{align*}
ABUV&  \ll U^{1- 1/2r}\(V^{1/2}p^{1/2r} + Vp^{1/4r}\) \sum_{\chi\in \cX^*}  \left| \sum_{a\in \cA}  \chi\( a  \) \right|  \left|  \sum_{b \in \cB}  \chi\(b\)\right|  \\
&  \ll U^{1- 1/2r} Vp^{1/4r} \sum_{\chi\in \cX^*}  \left| \sum_{a\in \cA}  \chi\( a  \) \right|  \left|  \sum_{b \in \cB} ,
 \chi\(b\)\right| , 
\end{align*}
which we simplify as 
\begin{equation}
\label{eq:Prelim} 
U^{1/2r} \le \frac{p^{1/4r}}{AB}  \sum_{\chi\in \cX^*}  \left| \sum_{a\in \cA}  \chi\( a  \) \right|  \left|  \sum_{b \in \cB}  \chi\(b\)\right|.
\end{equation}

Furthermore, using the orthogonality of characters, we obtain 
$$
  \sum_{\chi\in \cX^*}  \left| \sum_{a\in \cA}  \chi\( a  \) \right| ^2 
  \le   \sum_{\chi\in \cX}  \left| \sum_{a\in \cA}  \chi\( a  \) \right| ^2 = (p-1)A,
$$
and similarly for the sum over $b\in \cB$.
Hence, by the Cauchy inequality we have
\begin{align*}
\( \sum_{\chi\in \cX^*}  \left| \sum_{a\in \cA}  \chi\( a  \) \right|  \left|  \sum_{b \in \cB} 
 \chi\(b\)\right| \)^2 & \le 
  \sum_{\chi\in \cX^*}  \left| \sum_{a\in \cA}  \chi\( a  \) \right| ^2  \sum_{\chi\in \cX^*}  \left|  \sum_{b \in \cB} 
 \chi\(b\)\right| ^2\\ & \le p^2AB.
\end{align*}
Substituting this inequality in~\eqref{eq:Prelim}  and using that 
$$
U\gg H/R = Hp^{-1/2r}, 
$$
we obtain 
$$
\(Hp^{-1/2r}\)^{1/2r} \ll  \frac{p^{1/4r}}{AB} p (AB)^{1/2} =  \frac{p^{1+1/4r}}{(AB)^{1/2}}
$$
or 
 $$
Hp^{-1/2r}  \ll    \frac{p^{2r+1/2}}{(AB)^r},
$$
which concludes the proof in the case $\cH \cap (\cA\cdot\cB) = \emptyset$. 

\subsection{Bounding the size of Hilbert cubes avoiding reciprocal sum sets}
Let now $\cH \cap \(\cA + \cB\)^{-1} =\emptyset$. 
We proceed exactly as in the case of $\cH \cap (\cA\cdot\cB) = \emptyset$
with the same parameter $R$ and the sets $\cU$ and $\cV$. 
However, instead of~\eqref{eq:Sum Prod} we arrive to the equation 
$$
u+v =(a+b)^{-1}, \qquad a\in \cA, \ b \in \cB, \ u \in \cU, \ v \in \cV,
$$
which we transform into 
$$
(u+v)^{-1}  =a+b, \qquad a\in \cA, \ b \in \cB, \ u \in \cU, \ v \in \cV.
$$
This time we express the number $N$ of solutions to this equation, which still satisfies~\eqref{eq:Vanish} 
via exponential sums
\begin{align*}
N &= \sum_{a\in \cA}  \sum_{b \in \cB}  \sum_{u \in \cU}  \sum_{v \in \cV}
\frac{1}{p}\sum_{\lambda\in \F_p} \ep\(\lambda\((u+v)^{-1} -a -b\)\)\\
&= \frac{1}{p}\sum_{\lambda\in \F_p}   \sum_{a\in \cA}  \ep\(-\lambda a  \) \sum_{b \in \cB}  
 \ep\(-\lambda  b \)\sum_{u \in \cU}  \sum_{v \in \cV}  \ep\(\lambda(u+v)^{-1}\).
\end{align*}
We now separate the contribution from the term corresponding to $\lambda =0$ and then use
Lemma~\ref{lem: Double Exp Sum}  instead of Lemma~\ref{lem: Double Char Sum} 
and the orthogonality of exponential functions. Hence, we  arrive to the same bound as 
in Section~\ref{sec: H AB}. This  concludes the proof in the case $\cH \cap \(\cA + \cB\)^{-1} =\emptyset$. 

\section*{Acknowledgement}

The author would like to thank Ilya Shkredov for his comments and informing the author about~\cite{Mosh}. 

During the preparation of this work, the author was  supported  by the  Australian Research Council  Grant~DP170100786
and by the Natural Science Foundation of China Grant~11871317.

\end{document}